\documentclass[11pt,a4paper]{article}
\usepackage[T1]{fontenc}
\usepackage[utf8]{inputenc}
\usepackage[british]{babel}
\usepackage{csquotes}
\usepackage[stretch=20]{microtype}
\usepackage{mathtools}
\mathtoolsset{centercolon}
\usepackage{enumitem}
\usepackage{amssymb}
\usepackage{enumitem}
\usepackage{xcolor}
\usepackage{hyperref}
\hypersetup{
	colorlinks,
	linkcolor={red!50!black},
	citecolor={blue!50!black},
	urlcolor={blue!80!black}
}
\usepackage{amsthm}
\usepackage{geometry}
\usepackage{cleveref}

\makeatletter
\newcommand{\leqnomode}{\tagsleft@true}
\newcommand{\reqnomode}{\tagsleft@false}
\makeatother

\DeclarePairedDelimiter\norm{\lVert}{\rVert}

\renewcommand*{\P}{\mathbf{P}}
\DeclareMathOperator{\E}{\mathbf{E}}

\newcommand*\de{\mathop{}\!\mathrm{d}}
\DeclarePairedDelimiterXPP\Prob[1]{\mathrm{P}}{[}{]}{}{#1}
\newcommand*{\numberset}{\mathbb}
\newcommand*{\N}{\numberset{N}}

\newcommand*{\R}{\numberset{R}}

\newcommand{\oneover}[1]{\frac{1}{#1}}

\newcommand{\poly}{P_{n,d}^\beta}
\newcommand{\ball}{\mathbb{B}^d}

\newcommand*{\G}{\Gamma}

\newcommand*{\vol}{\mathrm{vol}}
\newcommand{\F}{\mathrm{F}}
\renewcommand{\c}{\mathrm{c}}
\newcommand{\I}{\mathrm{I}\;\!}
\newcommand{\const}{\frac{K}{\kappa_d}}
\renewcommand{\epsilon}{\varepsilon}
\renewcommand{\phi}{\varphi}
\theoremstyle{plain}
\newtheorem{theorem}{Theorem}[section]
\newtheorem{corollary}[theorem]{Corollary}
\newtheorem{lemma}[theorem]{Lemma}

\theoremstyle{definition}

\theoremstyle{remark}
\newtheorem{remark}{Remark}
\title{\textbf{Phase transition for the volume of \\ high-dimensional random polytopes}}
\author{Gilles Bonnet\thanks{\textit{Fakultät für Mathematik, Ruhr-Universität Bochum.} 
}, 
Zakhar Kabluchko\thanks{\textit{Institut für Mathematische Stochastik, Westfälische Wilhelms-Universität Münster.} 
}, 
and Nicola Turchi\thanks{\textit{Unité de Recherche en Mathématiques, Université du Luxembourg}.
}}
\date{}
\begin{document}
\maketitle

\begin{abstract}
    \noindent 
    The beta polytope $P_{n,d}^\beta$ is the convex hull of $n$ i.i.d.\ random points distributed in the unit ball of 
    \(\R^d\) according to a density proportional to $(1-\norm{x}^2)^{\beta}$ if $\beta>-1$ (in particular, \(\beta=0\) corresponds to the uniform distribution in the ball), or uniformly on the unit sphere if $\beta=-1$. We show that the expected normalized volumes of high-dimensional beta polytopes exhibit a phase transition and we describe its shape.
    We derive analogous results for the intrinsic volumes of beta polytopes and, when $\beta=0$, their number of vertices.
    
    \bigskip
    
    \noindent \textbf{Keywords}. Beta distribution, convex hull, expected volume, phase transition, random polytopes.
    
    \noindent \textbf{MSC 2010}.
    Primary 52A23; 
    Secondary 52A22, 
    52B11, 
    60D05. 
\end{abstract}

\section{Introduction}
In the last few years there has been an increasing interest in the study of high-dimensional random convex hulls, see \cite{BCGTT,bonnet2019facets,chakraborti2020note,dyer1992volumes,frieze2020random,gatzouras2009threshold,pivovarov2007volume} for example. Older works include \cite{barany1988shape,buchta1986conjecture,kingman1969random,miles1971isotropic}.
A common way to construct such objects is described by the following general principle. For every couple of natural numbers \(d\) and \(n\) with $n>d$, and any probability measure \(\mu\) on \(\R^d\) whose support is compact and not contained in a hyperplane, we consider the convex hull $P_{n,d}^{\mu}$ of $n$ i.i.d.\ random points distributed according to $\mu$.
Clearly $P_{n,d}^{\mu} \subseteq \mathrm{conv}(\mathrm{supp} (\mu))$ a.s., where $\mathrm{conv}$ denotes the convex hull operator and $\mathrm{supp} (\mu)$ the support of \(\mu\).
Therefore, if $n = n(d)$ and $\mu = \mu (d)$, then the sequence of expected normalized volumes
\begin{equation} \label{eq:ratio}
    \frac{ \E \vol_d (P_{n,d}^{\mu})}{ \vol_d ( \mathrm{conv}(\mathrm{supp} (\mu))) } , \quad d\in\{ 2, 3 , \ldots \},
\end{equation}
takes values in the interval $[0,1]$.

It is clear that if $n$ grows sufficiently fast, then the ratio above tends to $1$. On the other hand, in the slowest regime $n=d+1$ it goes to zero (see the appendix for a short proof of this claim).
Thus, between these extreme regimes there is a family of regimes where a transition happens. An interesting problem is to identify them. Namely, one would like to answer the following questions :
\begin{enumerate}[label={(Q\arabic*)}]
    \item \label{Q1} What are the regimes $n=n(d)$ where the transition happens ?
    \item \label{Q2} For such regimes, what is the shape of the transition ?
\end{enumerate}
The first question has been answered for a variety of families of distributions.
Dyer, F\"{u}redi and McDiarmid considered in \cite{dyer1992volumes} the setting where $\mu$ if the uniform discrete measure on the $2^d$ vertices of the unit cube $[-1,1]^d$ and showed that such regimes satisfy $n^{1/d}\to 2/\sqrt{e}$, meaning that if $n^{1/d} < 2/\sqrt{e}-\epsilon$ (resp. $>2/\sqrt{e}+\epsilon$), for all $d$ sufficiently large and for a fixed $\epsilon>0$, then the sequence \eqref{eq:ratio} tends to $0$ (resp. 1).
Gatzouras and Giannopoulos extended this result in \cite{gatzouras2009threshold} to the general setting where $\mu$ is the product $\mu_1^{\otimes d}$ of $d$ copies of a fixed symmetric measure $\mu_1$ with compact support in $\R$. The criterion in this setting has the same form, namely $n^{1/d}\to c$ where $c>0$ is a constant which can be expressed in term of $\mu_1$.
The picture changes when the convex hull of the support of the measure $\mu$ is the Euclidean unit ball $\ball$. It was already known from the work of Bárány and Füredi \cite{barany1988approximation},  in which they study the optimal approximation of the ball via deterministic polytopes, that the regimes where the transition happens must satisfy $n \geq d^{d/2}$.
In fact, Pivovarov proved \cite{pivovarov2007volume} that if $\mu$ is the uniform distribution on the Euclidean unit sphere $\mathbb{S}^{d-1}$, 
then the regimes answering \ref{Q1} satisfy 
\[
\frac{\log n} {d\log d} \to \frac{1}{2},
\]
i.e. the threshold is super-exponential.

The latter result has been extended in \cite{BCGTT} to the class of so-called beta distributions which have recently attracted a considerable attraction in stochastic geometry \cite{affentrager1991convex,bonnet2017monotonicity,giorgos, grote2019limit, gusakova2020delaunay, kabluchko2019angle, kabluchko2019angles,kabluchko2019recursive,kabluchko2019cones,kabluchko2019expected,kabluchko2018beta, miles1971isotropic,ruben1980canonical}. 
These distributions are described by their density, which is proportional to $(1-\norm{x}^2)^{\beta} \mathbf{1}(\norm{x}< 1)$, where $\norm{\cdot}$ denotes the Euclidean norm and where $\beta > - 1$ is a parameter which might depend on $d$. For example, \(\beta=0\) generates the uniform distribution on the Euclidean unit ball.
We extend this family of distributions by defining the beta distribution with $\beta=-1$ to be the uniform distribution on the unit sphere.
For such distributions, that we denote here by \(\mu_\beta\), it was shown in \cite{BCGTT} that \[
\frac{\log n  }{(d+2\beta) \log d} \to \frac{1}2 \]
is a criterion answering \ref{Q1}.
In particular, for any $\epsilon>0$,
\begin{equation} \label{eq:thresholdbeta}
    \lim_{d\to\infty} \frac{ \E\vol_d(P_{n,d}^\beta)}{ \vol_d ( \ball ) } = 
    \begin{cases}
        0 & \text{if } n \leq \exp \bigl( (1-\epsilon) (\frac{d}{2}+\beta)  \log d \bigr),
        \\1 & \text{if } n \geq \exp \bigl( (1+\epsilon) (\frac{d}{2}+\beta)  \log d \bigr),
    \end{cases}
\end{equation}
where we use \(\poly\) as a shorthand for \(P_{n,d}^{\mu_\beta}\).
This threshold phenomenon was obtained by generalizing Pivovarov's arguments. In the present paper we use instead an integral representation of the expected volume of beta polytopes obtained by Kabluchko, Temesvari and Thäle in \cite{kabluchko2019expected}, see Theorem \ref{thm:beta0} below for a statement of this  formula. By studying carefully this integral we will be able to refine \eqref{eq:thresholdbeta} by also answering \ref{Q2} for the beta polytopes model. 

Namely, we will show that whenever \(x\) is a positive constant and the number  \(n\) of random points grows like \(\exp\bigl((\frac{d}{2}+\beta)\log\frac{d}{2x}\bigr)\), then  
    \[\frac{\E\vol_d(\poly)}{\vol_d(\ball)}\to e^{-x}\in(0,1),\]
as \(d\) tends to infinity.

The rest of the paper is structured as follows. In the next section we introduce the required notation and present some preliminaries. In \Cref{sec:result} we state the main theorem and derive some simple but interesting corollaries. Finally, \Cref{sec:proofs} is dedicated to the proof of the theorem and its corollaries.

\section{Notation and preliminaries}
For an integer \(d\ge 2\), we denote by \(\ball\) the Euclidean unit ball in \(\R^d\) and by \(\mathbb{S}^{d-1}\) its boundary. We denote by \(\vol_d\) the Lebesgue measure on \(\R^d\) and we write \(\kappa_d\coloneqq\vol_d(\ball)= \pi^{d/2}/\Gamma(1+d/2)\).
For two sequences $a(d)$ and $b(d)$ we use the notation $a\sim b$, respectively $a=o(b)$, to mean that the ratio $a(d)/b(d)$ tends to $1$, respectively $0$, as $d\to\infty$.

Given a convex body \(K\subset \R^d\) and a number \(k\in\{0,\ldots, d\}\), the \(k\text{-th}\) intrinsic volume of \(K\), denoted by \(V_k(K)\), is the geometric functional defined by means of the Steiner formula, namely
\begin{equation}
    \label{eq:defintrinsic}
        \vol_d(K+t\mathbb{B}^d)=\sum_{k=0}^d t^{d-k}\kappa_{d-k}V_k(K),
\end{equation}
for every \(t\ge 0\). In particular, it holds that \(V_d=\vol_d\) and that \(V_{d-1}(K) \) is half the surface area of \(K\).
Moreover, whenever \(K\) is a polytope, i.e. the convex hull of a finite number of points, then we indicate with \(f_0(K)\) the number of its vertices.

\bigskip

The beta distribution with parameter \(\beta>-1\) is the continuous probability measure on \(\R^d\) with density 
\[
x\mapsto\frac{\Gamma(\frac{d}{2}+\beta+1)}{\pi^{\frac{d}{2}}\Gamma(\beta+1)}(1-\norm{x}^2)^{\beta},\quad \norm{x}<1,
\]
and \(0\) everywhere else. 
In particular, the beta distribution with \(\beta=0\) is just the uniform distribution in the Euclidean unit ball. We also say that the beta distribution with parameter \(\beta=-1\) is the uniform distribution on the Euclidean sphere \(\mathbb{S}^{d-1}\); this is justified by the fact that the beta distributions with parameters \(\beta>-1\) converge weakly to the uniform probability distribution on the sphere as \(\beta\to-1\) (a proof of this fact can be found in \cite{kabluchko2019expected}).

Let \(X_1,\ldots,X_n\) be i.i.d. random points in \(\R^d\) distributed according to the beta distribution with parameter \(\beta\).
We construct the beta polytope \(\poly\) as
\[
\poly\coloneqq\mathrm{conv}(\{X_1,\ldots,X_n\})\subset\ball.
\]
For \(z>0\), we introduce the non-negative quantities
    \begin{align*}
	    \c_{z}&\coloneqq\frac{\G(z+ 1/2)}{2\sqrt{\pi}\G(z+1)},\\
	    \quad\F_{z-1}(h)&\coloneqq 2z\,\c_z \int_{-1}^{h} (1-s^2)^{z-1}\de s,\quad h\in[-1,1].
    \end{align*}
Note that \(\F_{z-1}\) takes values in \([0,1]\), because \(\c_z\) is chosen such that \(\F_{z-1}(1)=1\).

An explicit representation of the expected volume of \(\poly\) was proved in \cite[Theorem 2.1 for the case \(\beta>-1\) and Corollary 3.9 for \(\beta=-1\)]{kabluchko2019expected} and can be stated as follows.
\begin{theorem}
\label{thm:beta0}
	Let
	    \[
	    K_{n,d}^\beta=\frac{(d+1)\kappa_d}{2^d\pi^{\frac{d+1}{2}}}\binom{n}{d+1}\Bigl(\beta+\frac{d+1}{2}\Bigr)\Biggl(\frac{\Gamma\bigl(\frac{d+2}{2}+\beta\bigr)}{\Gamma\bigl(\frac{d+3}{2}+\beta\bigr)}\Biggr)^{d+1}
	    \]
	 and \(q=(d+1)(\beta-\frac{1}{2})+\frac{d}{2}(d+3)\). Then,
	 \begin{equation}
	 \label{eq:originalvolume}
	     \E\vol_d(\poly)=K_{n,d}^\beta\int_{-1}^1\bigl(1-h^2\bigr)^q \,\F_{\beta+\frac{d-1}{2}}(h)^{n-d-1}\de h.
	 \end{equation}
\end{theorem}

\section{The main result and related corollaries} 
\label{sec:result}
We want to prove the following statement.
	\begin{theorem}
	\label{thm:main}
		Fix \(x\in(0,+\infty)\) and consider any sequence \(\beta=\beta(d)\ge -1\). Let \(n=n(d)\) be a sequence of natural numbers of the form
		\begin{equation}
		\label{eq:n}
			n=\Bigl(\frac{d}{2x+o(1)}\Bigr)^{d/2+\beta} ,
		\end{equation}
		where \(o(1)\) is a sequence converging to \(0\) as \(d\to\infty\).
		Then,
		\begin{equation} \label{eq:main-lim}
			\lim_{d\to\infty}\frac{\E\vol_d(\poly)}{\vol_d(\mathbb{B}^d)}=e^{-x}.
		\end{equation}
	\end{theorem}
\begin{remark} Taking into account both \Cref{thm:main} and the threshold established in \cite{BCGTT}, we can state that, for any sequences of natural numbers \(n=n(d)\) and real numbers \(\beta=\beta(d)\ge-1\),
\begin{equation*}
    \lim_{d\to\infty}\frac{\E\vol_d(\poly)}{\vol_d(\mathbb{B}^d)}=\lim_{d\to\infty}\exp\Bigl(-\frac{d}{2}\exp\Bigl(-\frac{2\log n}{d+2\beta} \Bigr)\Bigr),
\end{equation*}
whenever the right-hand side exists.
\end{remark}


The first consequence of \Cref{thm:main} that we mention is about the ratio of expected intrinsic volumes of \(\poly\).

\begin{corollary}
\label{cor:intrinsic}
    Under the same hypotheses of \Cref{thm:main} on \(x\) and \(\beta\), let \(k=k(d)\in\{1,\ldots, d\}\) be a diverging sequence of natural numbers as \(d\to\infty\) and \(n=n(d)\) be a sequence of natural numbers of the form
		\begin{equation}
			n=\Bigl(\frac{k}{2x+o(1)}\Bigr)^{d/2+\beta} ,
		\end{equation}
		where \(o(1)\) is a sequence converging to \(0\) as \(d\to\infty\). Then,
    \[
        \lim_{d\to\infty}\frac{\E V_k(\poly)}{ V_k (\mathbb{B}^d)}=e^{-x}.
    \]
\end{corollary}
     The second consequence is an asymptotics for the expected number of vertices of the polytope that arises as the convex hull of points which are picked independently and uniformly at random inside of the ball.
\begin{corollary} \label{cor:vertices}
    Under the same assumptions as in  \Cref{thm:main}, fix $\beta=0$. Then the asymptotic behavior of the expected number of vertices $P_{n,d}^0$ is given by
    \begin{equation*}
        \E f_0(P_{n,d}^0) \sim ( 1-e^{-x} )\,n ,
    \end{equation*}
    as \(d\to\infty\).
\end{corollary}

\section{Proofs}
\label{sec:proofs}
In view of formula \eqref{eq:originalvolume}, it is convenient for our purposes to change the variables \(d\) and \(n\) in the following way:
	\begin{align*}
	    D&\coloneqq\frac{d+1}{2}\in\{3/2,2,\ldots\},\\
	    N&\coloneqq n-d-1\in\{0,1,\ldots\}.
	\end{align*}
This makes the constant in front of the integral \eqref{eq:originalvolume} become
\begin{equation} 
\label{eq:K}
    K\coloneqq K_{N+2D,2D-1}^\beta=4D(D+\beta)\binom{N+2D}{2D}\c_{D+\beta}^{2D}\,\kappa_d,
\end{equation}
while the integral itself is given by
\begin{equation*}
   \int_{-1}^1 (1-h^2)^{2D(D+\beta)-1}\F_{D+\beta-1}(h)^N\de h.
\end{equation*}
We study the behavior of integrals of the form
\begin{equation}
    \label{eq:Iab}
   \quad \I_a^b =\I_a^b(N,D,\beta) \coloneqq \int_a^b (1-h^2)^{2D(D+\beta)-1}\F_{D+\beta-1}(h)^N\de h,
\end{equation}
for \(a,b\in[-1,1]\). Indeed, \Cref{thm:beta0} can be reformulated as
	\begin{equation}
	\label{eq:reformulation}
	\E \vol_d (\poly)=K\, \I_{-1}^{+1}.
	\end{equation}
In order to prove Theorem \ref{thm:main}, we need to show that the limit \eqref{eq:main-lim} holds for any fixed $x\in(0,\infty)$, any sequence $\beta(d)\geq -1$ and any sequence $n(d)$ satisfying condition \eqref{eq:n}.
In a first step we show that it is enough to consider sequences $n(d)$ satisfying a more restrictive condition than \eqref{eq:n} which is stated in term of $N=n-d-1$ and $D=(d+1)/2$.
\begin{lemma}
\label{lem:reduction}
	It is sufficient to prove \Cref{thm:main} for all the integers \(N=N(D)\) such that
	\begin{equation*}
		N \sim \Bigl(\frac{D}{x}\Bigr)^{D+\beta} \sqrt{1+\frac{\beta}{D}},
	\end{equation*}
	as \(D\to\infty\), with \(x\in(0,+\infty)\) and \(\beta=\beta(D)\in[-1,+\infty)\).
\end{lemma}
\begin{proof}
    Assume that if \(N\sim D^{D+\beta}x^{-D-\beta}\sqrt{1+\beta/D}\) then, for \(d=2D-1\) and \(n=N+d+1\), it holds
   	\begin{equation*}
			\lim_{d\to\infty}\frac{\E\vol_d(\poly)}{\vol_d(\mathbb{B}^d)}=e^{-x}.
		\end{equation*}
    Notice that if \(n\) is as in equation \eqref{eq:n}, then 
    \[ N = \left(\frac{2D-1}{2x+o(1)}\right)^{D+\beta-1/2} -2D
    = \left(\frac{D}{x+o(1)}\right)^{D+\beta-1/2} , \]
    where the equalities above follow, respectively, from the definition of $N$ and $D$, and trivial simplifications. 
    Therefore
    \[ N = \left(\frac{D}{x+o(1)} \left( \sqrt{\frac{ x/D }{ 1+\beta/D } }\right)^{1/(D+\beta)} \right)^{D+\beta} \sqrt{1+\frac{\beta}{D}}
    = \left(\frac{D}{x+o(1)}\right)^{D+\beta} \sqrt{1+\frac{\beta}{D}} . \]
    Now, fix an arbitrary small constant \(\epsilon>0\). There exists \(D\) large enough such that:
        \begin{equation*}
            N^-\coloneqq \biggl\lfloor\Bigl(\frac{D}{x+\epsilon}\Bigr)^{D+\beta} \sqrt{1+\frac{\beta}{D}}\biggr\rfloor\le N \le \biggl\lceil\Bigl(\frac{D}{x-\epsilon}\Bigr)^{D+\beta} \sqrt{1+\frac{\beta}{D}}\biggr\rceil\eqqcolon N^+.
        \end{equation*}
    By assumption, since \(x+\epsilon\in(0,+\infty)\) is a fixed constant, if \(n^{-}\coloneqq N^{-}\!+d+1\) then
    \begin{equation*}
			\lim_{d\to\infty}\frac{\E\vol_d(P^{\beta}_{n^{-},d})}{\vol_d(\mathbb{B}^d)}=e^{-x-\epsilon}.
		\end{equation*}
    Analogously, using \(N^+\) yields \(e^{-x+\epsilon}\) as the limit. This allows us to deduce the claimed statement, because of the monotonicity of \(n\mapsto\E\vol_d(\poly)\) for any \(d\) fixed (note that \(\poly\subseteq P^{\beta}_{n+1,d}\) a.s.), the arbitrariness of \(\epsilon\) and the continuity of the exponential map.
\end{proof}
In view of \eqref{eq:reformulation} and \Cref{lem:reduction}, we want to show that
	\begin{equation}
	\label{eq:totallimit}
		\lim_{d\to\infty}\frac{K \I_{-1}^{+1}}{\kappa_d}=e^{-x},
	\end{equation}
whenever
	\begin{equation}
	\label{eq:N}
		N\sim\Bigl(\frac{D}{x}\Bigr)^{D+\beta}\sqrt{1+\frac{\beta}{D}}.
	\end{equation}
Our strategy is to find convenient \(a=a(d)\), \(b=b(d)\in(0,1)\) in such a way that 
	\begin{equation*}
			\lim_{d\to\infty}\frac{K \I_{-1}^{a}}{\kappa_d}= 0,\qquad\lim_{d\to\infty}\frac{K \I_{a}^{b}}{\kappa_d}= e^{-x}\qquad\lim_{d\to\infty}\frac{K \I_{b}^{1}}{\kappa_d}= 0.
	\end{equation*}
	
In order to do so, we first gather some useful facts in the next four Lemmas. We start with an inequality on the ratio of Gamma functions due to Wendel \cite{Wendel}.
	\begin{lemma}
	\label{lem:Gautschi}
	For \(z>0\) it holds that
	\begin{equation*}
	   \oneover{\sqrt{z+1/2}}\le \frac{\G(z+1/2)}{\G(z+1)}\le\oneover{\sqrt{z}}.
	\end{equation*}
	\end{lemma}
	By the previous lemma, 
	it holds
	\begin{equation*}
	\mathrm{c}_{D+\beta}\le \oneover{2\sqrt{\pi (D+\beta)}} 	\qquad\text{ and }\qquad 
	\mathrm{c}_{D+\beta}\sim\oneover{2\sqrt{\pi (D+\beta)}},	\end{equation*} as \(D\to\infty\). Notice that this implies that if \(N\) is chosen as in \Cref{lem:reduction}, then
	\begin{equation} \label{eq:NcDb}
	    2\sqrt{\pi} N\c_{D+\beta}\sim D^{D+\beta-\frac{1}{2}} x^{-D-\beta}.
	\end{equation}
	    
	\begin{lemma}
	    \label{lem:expineq}
	    For all \(m\geq1\) and all \(y\in(-\infty, m)\):
	    \begin{equation}
	    \label{eq:expineq}
	        \exp\Bigl(-\frac{y^2}{m-y}\Bigr)\le e^y\Bigl(1-\frac{y}{m}\Bigr)^m\le 1.
	    \end{equation}
        In particular, if \(y=y(m)\) is such that \(y^2/m\to 0\) as \(m\to\infty\), then
        \begin{equation}
            \label{eq:expasy}
            \Bigl(1-\frac{y}{m}\Bigr)^m\sim e^{-y},
        \end{equation}
    as \(m\to\infty\).
	\end{lemma}
	\begin{proof}
	    It is well known that
	    \begin{equation*}
	        \frac{z}{1+z}\le\log(1+z)\le z,\qquad z\in(-1,\infty).
	   \end{equation*}
	  We use it with \(z=-y/m>-1\). Applying the increasing mapping \(\exp(m\cdot)\) on both sides of the inequality we get:
	  \begin{equation*}
	       \exp\Bigl(-\frac{my}{m-y}\Bigr)\le\Bigl(1-\frac{y}{m}\Bigr)^m\le\exp(-y).
	   \end{equation*}
	 The fact that \(\frac{my}{m-y}=y+\frac{y^2}{m-y}\) yields the result.
	\end{proof}
\begin{lemma}
\label{lem:constant}
    For \(D\) large enough and \(N\) as in \Cref{lem:reduction}, \(\const\le (D/x)^{2D(D+\beta)}\).
\end{lemma}
\begin{proof} 
    Recall from \eqref{eq:K} that
    \[ \frac{K}{\kappa_d} = 4D(D+\beta)\binom{N+2D}{2D}\c_{D+\beta}^{2D} . \] 
    We start by bounding the binomial coefficient.
    Using the bounds \( \binom{k+\ell}{k} \leq (k+\ell)^k / k ! \) and \(k!\geq (k/e)^k\), with \(k=2D\) and \( \ell=N \), we note that 
    \begin{equation*}
        \binom{N+2D}{2D}\le \Bigl(\frac{e(N+2D)}{2D}\Bigr)^{2D} = \Bigl(\frac{e N}{2D}\Bigr)^{2D} e^{o(1)},
    \end{equation*}
    where the last equality is due to the RHS inequality of \cref{eq:expineq}. Indeed
    \begin{equation*}
        1 \leq \Bigl(\frac{N+2D}{N}\Bigr)^{2D} = \Bigl( 1 + \frac{o(1)}{2D}\Bigr)^{2D} \leq e^{o(1)} . 
    \end{equation*}
    Also by \Cref{lem:Gautschi}, we know that $c_{D+\beta} \leq 1/(2\sqrt{\pi (D+\beta)})$.
    
    Thus for \(N=(D/x)^{D+\beta}e^{o(1)}\sqrt{1+\beta/D}\) we get
    \begin{equation*}
    \begin{split}
		\frac{K}{\kappa_d} 
		& \leq 4D(D+\beta) \Bigl(\frac{e (D/x)^{D+\beta}e^{o(1)}}{4D \sqrt{\pi (D+\beta)}}\Bigr)^{2D} e^{o(1)} \\
		& =\Bigl(\frac{e^{1+o(1)}}{4D \sqrt{\pi}}\Bigr)^{2D} \frac{4D}{(D+\beta)^{D-1}}\Bigl(\frac{D}{x}\Bigr)^{2D(D+\beta)}
		\leq \Bigl(\frac{D}{x}\Bigr)^{2D(D+\beta)},
	\end{split}
    \end{equation*}
    where the last bound holds for $D$ large enough.
    This gives the conclusion.
\end{proof}
The following inequality provides a useful approximation of the functions \(\F_{z-1}\).
\begin{lemma} \label{lem:bounds}
	\label{eq:boundsF}
	For any \(h\in(0,1)\) and \(z>0\),
	\begin{equation*}
		1-\frac{1-h^2}{2h^2(z+1)}\le  \frac{h(1-\F_{z-1}(h))}{\c_z(1-h^2)^{z}}\le 1.
	\end{equation*}
\end{lemma}
\begin{proof}
	Recall that for \(h\in(0,1)\), the function  \( \F_{z-1} (h) \) is defined by
	\[ \F_{z-1}(h) = 2z\,\c_{z} \int_{-1}^{h} (1-s^2)^{z-1}\de s \]
	and in particular for $h=1$ this simplifies as \( \F_{z-1}(1) = 1 \).
	Thus
	\[ \frac{1-\F_{z-1}(h)}{\c_{z}} 
	= 2 z \int_h^1 (1-s^2)^{z-1} \de s . \]
	We apply the substitution $t=\frac{s^2-h^2}{1-h^2}$ to get
	\[
	    \frac{1-\F_{z-1}(h)}{\c_{z}}
        = 2 z \int_0^1 [1-h^2- (1-h^2) t ]^{z-1} \frac{1-h^2}{2\sqrt{h^2 + (1-h^2) t }} \de t .
	\]
    Multiplying by $h / (1-h^2)^{z} $ gives
	\begin{equation} \label{aftersubstitution}
        \frac{h (1-\F_{z-1}(h) )}{\c_{z} (1-h^2)^{z} }
        = z \int_0^1 (1-t)^{z-1} \frac{1}{\sqrt{1 + \frac{1-h^2}{h^2} t }} \de t .
	\end{equation}
	The fraction in the last integrand makes the computation delicate and this is the step where we introduce the approximation given by
	\begin{equation} \label{PolyBounds}
      1 - \frac{1-h^2}{2 h^2} t 
      \leq \frac{1}{\sqrt{1 + \frac{1-h^2}{h^2} t }} 
      \leq 1 ,
      \qquad h,t \in (0,1).
    \end{equation}
	The lower bound follows from the fact that for any $ u \geq 0 $, $ (1+u)^{-1/2} \geq 1 - u/2 $.
	
	We get the upper bound of Lemma \ref{lem:bounds} by plugging the upper bound of \eqref{PolyBounds} in \eqref{aftersubstitution} and using the fact that
    $ \int_0^1 (1-t)^{z-1} \de t = 1/z $.
    
    It remains to show the lower bound of the lemma.
    Using the lower bound of \eqref{PolyBounds} and \eqref{aftersubstitution}, we get
    \begin{equation*}
        \frac{h (1-\F_{z-1}(h) )}{\c_z (1-h^2)^z }
        \ge z \int_0^1 (1-t)^{z-1} \de t - \frac{1-h^2}{2 h^2} z \int_0^1 (1-t)^{z-1} t  \de t.
	\end{equation*}
	The two last integrals evaluate nicely.
	As we have already mentioned, the first one equals $1/z$.
	The second one is the beta function $B(z,2) = \Gamma(z) \Gamma(2) / \Gamma(z+2) = 1 / [z(z+1)]$.
	The lower bound of lemma follows directly.
\end{proof}
We are now ready to prove the main result.
\begin{proof}[Proof of \Cref{thm:main}]
	Choosing \(z=D+\beta\) in \Cref{lem:bounds}, we get that for any \(a,b\in(0,1)\) with \(a<b\), and any \(h\in[a,b]\), it holds that
	\begin{equation*}
		1-A\c_{D+\beta}(1-h^2)^{D+\beta}\le \F_{D+\beta-1}(h)\le 1-B\c_{D+\beta} (1-h^2)^{D+\beta},
	\end{equation*}
	where \(A\coloneqq\oneover{a}\) and \(B\coloneqq \oneover{b}(1-\frac{1-a^2}{2a^2 (D+\beta)})\).
	This inequality can be used to bound \(\I_a^b\) on both sides. On the left-hand side we perform the change of variable 
	\begin{equation*}
		t=NA\c_{D+\beta}(1-h^2)^{D+\beta}.
	\end{equation*}
	The Jacobian of this transformation is such that
	\begin{equation*}
	   (1-h^2)^{2D(D+\beta)-1}\de h = -\frac{(1-h^2)^{(2D-1)(D+\beta)}}{2h(D+\beta)NA\c_{D+\beta}}\de t=
	   \frac{t^{2D-1}}{-2h(D+\beta)(NA\c_{D+\beta})^{2D}}\de t .
	\end{equation*}
	We perform an analogous change of variable on the right-hand side, this time replacing \(A\) by \(B\).  We also use \(1/b\le 1/h\le 1/a\). Hence, we get
	\begin{multline*}
		\frac{1}{2b(D+\beta) (NA\c_{D+\beta})^{2D}}\int_{NA\c_{D+\beta}(1-b^2)^{D+\beta}}^{NA\c_{D+\beta}(1-a^2)^{D+\beta}} t^{2D-1}\Bigl(1-\frac{t}{N}\Bigr)^N\de t\\
		\le \I_a^b \le \\
			\frac{1}{2a(D+\beta) (NB\c_{D+\beta})^{2D}}\int_{NB\c_{D+\beta}(1-b^2)^{D+\beta}}^{NB\c_{D+\beta}(1-a^2)^{D+\beta}} t^{2D-1}\Bigl(1-\frac{t}{N}\Bigr)^N\de t,
	\end{multline*}
	therefore, multiplying all terms by \(K/\kappa_d\), we get
	\begin{multline}
	\label{eq:step}
			\frac{1}{bA^{2D}}\binom{N+2D}{2D}\frac{2D}{N^{2D}}\int_{NA\c_{D+\beta}(1-b^2)^{D+\beta}}^{NA\c_{D+\beta}(1-a^2)^{D+\beta}} t^{2D-1}\Bigl(1-\frac{t}{N}\Bigr)^N\de t\\
		\le \frac{K\I_a^b}{\kappa_d} \le \\
		\frac{1}{aB^{2D}}\binom{N+2D}{2D}\frac{2D}{N^{2D}}\int_{NB\c_{D+\beta}(1-b^2)^{D+\beta}}^{NB\c_{D+\beta}(1-a^2)^{D+\beta}} t^{2D-1}\Bigl(1-\frac{t}{N}\Bigr)^N\de t.
	\end{multline}
	By \Cref{lem:expineq} and the fact that the mapping \(t\mapsto 	e^{-\frac{t^2}{N-t}}\) is decreasing for \(t\in(0,N)\), we get that for  \(t\in[NA\c_{D+\beta}(1-b^2)^{D+\beta},NA\c_{D+\beta}(1-a^2)^{D+\beta}]\),
	\begin{equation*}
	r_{N,D,a}\coloneqq \exp\Bigl(-N\cdot\frac{A^2\c_{D+\beta}^2(1-a^2)^{2(D+\beta)}}{1-A\c_{D+\beta}(1-a^2)^{D+\beta}}\Bigr) \le e^{t} \Bigl(1-\frac{t}{N}\Bigr)^N\le 1,
	\end{equation*}
	provided that \(A\c_{D+\beta}(1-a^2)^{D+\beta}<1\), which will be justified later with a particular choice of \(a\) and \(b\).
	
	Note that
	\begin{equation*}
	   1\le (2D-1)!\binom{N+2D}{2D}\frac{2D}{N^{2D}}=\frac{(N+2D)!}{N! N^{2D}}\le \Bigl(1+\frac{2D}{N}\Bigr)^{2D}=:s_{N,d}\,.
	\end{equation*}
    Hence, we get from \eqref{eq:step} 
	\begin{multline}
	\label{eq:integralbounds}
		\frac{r_{N,D,a}}{bA^{2D}}\oneover{(2D-1)!}\int_{NA\c_{D+\beta}(1-b^2)^{D+\beta}}^{NA\c_{D+\beta}(1-a^2)^{D+\beta}} t^{2D-1}e^{-t}\de t\\
			\le \frac{K\I_a^b}{\kappa_d} \le \\
		\frac{s_{N,D}}{aB^{2D}}\oneover{(2D-1)!}\int_{NB\c_{D+\beta}(1-b^2)^{D+\beta}}^{NB\c_{D+\beta}(1-a^2)^{D+\beta}} t^{2D-1}e^{-t}\de t.
	\end{multline}
	Now note that \(t^{2D-1}e^{-t}/(2D-1)!\) is the probability density function of random variable with a \(\G(2D,1)\) probability distribution, hence also of a sum of \(2D\) independent copies \((E_i)_{i=1}^{2D}\) of an exponentially distributed random variable with mean \(1\). Therefore, by the weak law of large numbers, 
	\begin{equation*}
	\begin{split}
		&\lim_{D\to\infty} \oneover{(2D-1)!}\int_{NA\c_{D+\beta}(1-b^2)^{D+\beta}}^{NA\c_{D+\beta}(1-a^2)^{D+\beta}} t^{2D-1}e^{-t}\de t\\
			&\qquad\qquad=\lim_{D\to\infty}\P\Bigl(NA\c_{D+\beta}(1-b^2)^{D+\beta}\le \sum_{i=1}^{2D} E_i\le NA\c_{D+\beta}(1-a^2)^{D+\beta} \Bigr)=1
	\end{split}
	\end{equation*}
	granted that there exists an \(\epsilon>0\) such that for \(D\) large enough
	\begin{equation}
	\label{eq:conditionWLLN}
	\begin{split}
		NA\c_{D+\beta}(1-a^2)^{D+\beta}&>(1+\epsilon)2D,\\
		NA\c_{D+\beta}(1-b^2)^{D+\beta}&<(1-\epsilon)2D.
	\end{split}
	\end{equation}
	Now assume without loss of generality that \(D>x\) and define the following quantities in \((0,1)\):
	\begin{equation}
	\label{eq:precisebell}
	\begin{split}
		a&\coloneqq \sqrt{1-\frac{x}{D}\Bigl(1+\frac{\log\bigl( D^2(D+\beta)\bigr)}{D+\beta}\Bigr)},\\
		b&\coloneqq  \sqrt{1-\frac{x}{D}}.
	\end{split}
	\end{equation}
	Since \(a^{2D}= \bigl(1-\frac{x}{D}\bigl(1+\frac{\log( D^2(D+\beta))}{D+\beta}\bigr)\bigr)^D\) and \(x\bigl(1+\frac{\log( D^2(D+\beta))}{D+\beta}\bigr)\to x\) as \(D\to\infty\) then, by \eqref{eq:expasy},
	\begin{equation*}
	a^{2D}\sim e^{-x}.
	\end{equation*}
	Analogously, \(b^{2D}\sim e^{-x}, \) as \(D\to\infty\).
	Moreover
	\[
	(1-a^2)^{D+\beta}=\Bigl(\frac{x}{D} \Bigr)^{D+\beta}\Bigl(1+\frac{\log\bigl( D^2(D+\beta)\bigr)}{D+\beta}\Bigr)^{D+\beta}.
	\]
	Using \(m=D+\beta\) and \(y=-\log\bigl(D^2(D+\beta)\bigr)\) in \Cref{lem:expineq}, since \(y^2/m\to 0\) as \(D\to\infty\), we get
	\[
	\Bigl(1+\frac{\log\bigl( D^2(D+\beta)\bigr)}{D+\beta}\Bigr)^{D+\beta}\sim e^{\log (D^2(D+\beta))}=D^2(D+\beta),
	\]
	as \(D\to\infty\). 
	
	Summing up what we just computed above, we get for \(D\to\infty\),
	\begin{equation}
		\label{eq:importantasy}
	    \begin{split}
	(1-a^2)^{D+\beta}&\sim x^{D+\beta} D^{2-D-\beta}(D+\beta),\\
	(1-b^2)^{D+\beta}&=  x^{D+\beta} D^{-D-\beta}.
	    \end{split}
	\end{equation}
	Recall that \(A=1/a\to 1\) and \(N\c_{D+\beta}\sim D^{D+\beta-\frac{1}{2}} /(2\sqrt{\pi}x^{D+\beta}) \), as noted in equation \eqref{eq:NcDb}. Thus,  
	\begin{equation}
	    \label{eq:NAcDb}
	    \begin{split}
	    NA\c_{D+\beta}(1-a^2)^{D+\beta}&\sim \frac{1}{2\sqrt{\pi}} D^{3/2}(D+\beta),\\
    	NA\c_{D+\beta}(1-b^2)^{D+\beta}&\sim \frac{1}{2\sqrt{\pi}} D^{-1/2},
    	\end{split}
	\end{equation}
	and, in particular, the assumptions \eqref{eq:conditionWLLN} are verified. Note also that we see that we also verified that \(A c_{D+\beta} (1 - h^2)^{D+\beta}<1\) as previously claimed.
	
	Moreover, in such a case, as \(D\to\infty\),
	\begin{equation*}
	r_{N,D,a}	A^{-2D}= r_{N,D,a} a^{2D} \sim 1\cdot e^{-x}= e^{-x},
	\end{equation*}
	where the asymptotics \( r_{N,D,a}\sim 1\) is a consequence of \eqref{eq:importantasy}, indeed by definition \(-\log r_{N,D,a}\!=NA^2\c_{D+\beta}^2(1-a^2)^{2(D+\beta)}/(1-A\c_{D+\beta}(1-a^2)^{D+\beta})\to 0\).
	
	We proved that if assumptions \eqref{eq:N} and \eqref{eq:precisebell} are satisfied then the left-hand side of \eqref{eq:integralbounds} tends to \(e^{-x}\). It is a simple check to see that everything holds in the same way for the right-hand side, since we only replace \(A\) by \(B\). Indeed, as far as the prefactor of the integral is concerned,
		\begin{equation*}
	s_{N,D}B^{-2D}=\Bigl(1+\frac{2D}{N}\Bigr)^{2D}b^{2D}\Bigl(1-\frac{1-a^2}{2a^2(D+\beta)}\Bigr)^{-2D}\sim 1\cdot e^{-x}\cdot 1=e^{-x},
	\end{equation*} 
	since \(\frac{1-a^2}{2a^2}=o(1)\). Hence, both the upper and the lower bound for \(K\I_a^b/\kappa_d\) have the same asymptotics, which allows to conclude that
	
	\begin{equation*}
	\frac{K\I_a^b}{\kappa_d}\sim e^{-x},
	\end{equation*}
	as \(d\to \infty\). 
	
    With Lemmas \ref{lem:a} and \ref{lem:b} below we have that under the same conditions, 
	\begin{equation*}
	\frac{K\I_{-1}^a}{\kappa_d}\to 0\quad\text{and}\quad 	\frac{K\I_b^1}{\kappa_d}\to 0,
	\end{equation*}
as \(d\to\infty\), which yields equation \eqref{eq:totallimit}.
\end{proof}
\begin{lemma} \label{lem:a}
		If \(a\) is chosen as in \eqref{eq:precisebell}, then \(\frac{K\I_{-1}^a}{\kappa_d}\to 0\) as \(d\to\infty\).
\end{lemma}
\begin{proof}
	Since \(\F_{D+\beta-1}\) is increasing, then
	\begin{equation*}
	\begin{split}
		\I_{-1}^a&\le \int_{-1}^1(1-h^2)^{2D(D+\beta)-1}\de h\, \F_{D+\beta-1}(a)^N\\ &=\sqrt{\pi}\frac{\G(2D(D+\beta))}{\G(2D(D+\beta)+1/2)}\F_{D+\beta-1}(a)^N\le \F_{D+\beta-1}(a)^N,
	\end{split}
	\end{equation*}
where the last equation holds for every \(D\) large enough, due to \Cref{lem:Gautschi}.
By \Cref{eq:boundsF}, 
	\begin{equation*}
		\F_{D+\beta-1}(a)\le 1-\frac{\c_{D+\beta}}{a}(1-a^2)^{D+\beta}\Bigl(1-\frac{1-a^2}{2a^2(D+\beta+1)}\Bigr)
	\end{equation*}
	so that
	\begin{equation}
	\label{eq:boundFN}
			\F_{D+\beta-1}(a)^N\le \exp\Bigl(-N\frac{\c_{D+\beta}}{a}(1-a^2)^{D+\beta}\Bigl(1-\frac{1-a^2}{2a^2(D+\beta+1)}\Bigr)\Bigr).
	\end{equation}
	Recall that $a\to 1$, thus \eqref{eq:NAcDb} gives us that
	\begin{equation*}
	N\frac{\c_{D+\beta}}{a}(1-a^2)^{D+\beta}\Bigl(1-\frac{1-a^2}{2a^2(D+\beta+1)}\Bigr)\sim \frac{1}{2\sqrt{\pi}}D^{3/2}(D+\beta) .
	\end{equation*}
Combining the above estimates with the bound of \(K/\kappa_d\) obtained in \Cref{lem:constant}, we get that, for an arbitrary positive constant $c$ less than $1/(2\sqrt{\pi})$ and $D$ large enough,
	\[ \frac{K \I_{-1}^a}{\kappa_d} \leq 
	\exp \Bigl( - D (D+\beta) \Bigl( c D^{1/2} - 2 \log \Bigl(\frac{D}{x}\Bigr) \Bigr) \Bigr) , \]
	which goes to zero. This concludes the proof.
	\end{proof}
\begin{lemma} \label{lem:b}
	If \(b\) is chosen as in \eqref{eq:precisebell}, then \(\frac{K\I_b^1}{\kappa_d}\to 0\) as \(d\to\infty\).
\end{lemma}
	\begin{proof}
		We bound \(\F_{D+\beta-1}\le 1\) so that 
		\begin{equation*}
		\begin{split}
			\I_b^1&\le\int_b^1 (1-h^2)^{2D(D+\beta)-1}\de h=\int_0^{1-b^2}\frac{s^{2D(D+\beta)-1}}{2\sqrt{1-s}}\de s\\
			&\le\frac{(1-b^2)^{2D(D+\beta)}}{4bD(D+\beta)}
			= \oneover{4bD(D+\beta)}\Bigl(\frac{x}{D}\Bigr)^{2D(D+\beta)},
		\end{split}
		\end{equation*}
			where we used the fact that \(1-b^2= x/D\).
	By \Cref{lem:constant} it follows that
		\begin{equation*}
			\const	\I_b^1\le\oneover{4bD(D+\beta)},
		\end{equation*}
	which tends to \(0\) as \(D\) tends to infinity.
		\end{proof}
\begin{proof}[Proof of \Cref{cor:intrinsic}]
    First of all, note that from the definition of intrinsic volume we get that 
    \[
        V_k(\mathbb{B}^d)=\binom{d}{k}\frac{\kappa_d}{\kappa_{d-k}},
    \]
    which can be deduced by plugging \(K=\mathbb{B}^d\) into \eqref{eq:defintrinsic}.
    Moreover, Proposition 2.3 in \cite{kabluchko2019expected} states that
    
    \[
    \E V_k(P_{n,d}^\beta)=\binom{d}{k}\frac{\kappa_d}{\kappa_k\kappa_{d-k}}\E \vol_k(P_{n,k}^{\beta'}),
    \]
    with \(\beta'\coloneqq\frac{d-k}{2}+\beta\). Together with the previous identity this implies that
    \[
        \frac{\E V_k(\poly)}{ V_k (\mathbb{B}^d)}=  \frac{\E \vol_{d'}(P_{n,d'}^{\beta'})}{\vol_{d'} (\mathbb{B}^{d'})},
    \]
    with \(d'\coloneqq k\). We can thus apply \Cref{thm:main} to the RHS with
    \[
        	n=\Bigl(\frac{d'}{2x+o(1)}\Bigr)^{d'\!/2+\beta'}=\Bigl(\frac{k}{2x+o(1)}\Bigr)^{d/2+\beta},
    \]
    which proves the claim, since \(d'\) also diverges by hypotheses.
\end{proof}
\begin{proof}[Proof of \Cref{cor:vertices}]
    The expected number of vertices of \(P_{n,d}^0 \) is given by \[
    \E f_0(P_{n,d}^0)=n\,\P(X_n \text{ is a vertex of }P_{n,d}^0).\]
    Observe that the latter event is the complement of $\{ X_n \in P_{n-1,d}^0\}$, whose probability is $\E \vol_d (P_{n-1,d}^0) / \vol_d (\ball)$ because $X_n$ is uniformly distributed, since $\beta=0$.
    Collecting these observations provides the so-called Efron's identity, see \cite{efron1965convex},
    \[ \frac{\E f_0(P_{n,d}^0)}{n} = 1 -  \frac{\E \vol_d (P_{n-1,d}^0)}{\vol_d (\ball)} . \]
    Combining this equation with the conclusion of Theorem \ref{thm:main} yields the proof.
\end{proof}
\section*{Appendix: Random simplices have small volumes}
In this section, we justify the claim made in the introduction that 
\begin{equation*}
    \lim_{d\to\infty}\frac{ \E \vol_d (P_{d+1,d}^{\mu})}{ \vol_d ( \mathrm{conv}(\mathrm{supp} (\mu))) } = 0 ,
\end{equation*}
independently of the choice of probability measures $\mu = \mu(d)$.

\begin{lemma}
For all natural numbers \(n> d\),
\begin{equation}
\label{eq:Mnd}
    M_{n,d} \coloneqq \sup \biggl\{ \frac{ \E \vol_d (P_{n,d}^{\mu})}{ \vol_d ( \mathrm{conv}(\mathrm{supp} (\mu))) } \biggr\}
    \le \frac{1}{2^{n-1}} \sum_{k=d}^{n-1} \binom{n-1}{k},
\end{equation}
where the supremum is taken on the set of all probability measures \(\mu\) in $\R^d$ whose support is compact and not contained in a hyperplane.
In particular, we see immediately that
\begin{equation*}
    M_{d+1,d} \le 2^{-d},
\end{equation*}
\end{lemma}
\begin{proof}
Let $\mu$ be a probability measure whose support is compact and not contained in a hyperplane. We need to bound the ratio appearing in the supremum in Equation \eqref{eq:Mnd}.
Without loss of generality we can assume that \(\mathrm{conv}(\mathrm{supp}(\mu))\) is a convex body \(K\) of unit volume. It is known (see e.g. \cite[Example 8.1.6]{bogachev2007measure}) that the set of probability measures of the form 
\[
\nu=\sum_{i=1}^m c_i \delta_{x_i},\quad m\in\N,\, c_i>0,\, \sum_{i=1}^m c_i=1,\, x_i\in K,
\]
where \(\delta_{x_i}\) is the Dirac measure concentrated at \(x_i\), is dense in the set of probability measures on \(K\) equipped with the weak topology. Moreover, each of these \(\nu\) admits an approximation via absolutely continuous probability measures with respect to the Lebesgue measure \(\de x\). For example, for all \(\epsilon>0\), the measure \(\nu_\epsilon\) such that
\[
\frac{\de \nu_\epsilon}{\de x}(x)=\sum_{i=1}^m \frac{c_i}{\vol_d((x_i+\epsilon\ball)\cap K)}\mathbf{1}(x\in\{(x_i+\epsilon\ball)\cap K\})
\]
is such an approximation. 
Since the expected volume of a random polytope \(P_{n,d}^\mu\) is a continuous function of $\mu$ with respect to the weak topology, we have that
\begin{equation*}
    M_{n,d} = \sup \bigl\{ \E \vol_d (P_{n,d}^{\mu}):\vol_d ( \mathrm{conv}(\mathrm{supp} (\mu)))=1, \mu \ll \de x \bigr\} .
\end{equation*}
By switching the order of integration, the expected volume of the random polytope $P_{n,d}^{\mu}$ can be written as follows,
\begin{equation}
\label{eq:PxinP}
    \E \vol_d (P_{n,d}^{\mu}) 
    = \int_ K \P ( x \in P_{n,d}^{\mu} ) \de x .
\end{equation}
By shifting $x$ and $\mu$ by the vector $(-x)$ the last integrand is the probability that the origin is contained in the random convex hull $P_{n,d}^{\mu-x}$.
Since $\mu$ is absolutely continuous we can apply an inequality of Wagner and Welzl \cite{wagner2001continuous}, which extends a famous identity due to Wendel \cite{wendel1962problem} valid in the symmetric setting, and tells us that
\begin{equation}
    \P ( x \in P_{n,d}^{\mu} )
    = \P ( 0 \in P_{n,d}^{\mu-x} )
    \leq \frac{1}{2^{n-1}} \sum_{k=d}^{n-1} \binom{n-1}{k}.
\end{equation}
Hence, we can bound \eqref{eq:PxinP} as
\begin{equation*}
    \E \vol_d (P_{n,d}^{\mu}) \le \oneover{2^{n-1}}\sum_{k=d}^{n-1} \binom{n-1}{k},
\end{equation*}
which proves the claim.
\end{proof}
Note that the RHS of \eqref{eq:Mnd} is the probability that the sum of \(n-1\) independent copies of a Bernoulli random variable with parameter \(1/2\) is greater than or equal to $d$. Therefore, it is a consequence of the weak law of large numbers that
\begin{equation*}
   \lim_{d\to\infty} M_{n,d} = 0, 
\end{equation*}
whenever \(n \leq (2-\epsilon)\,d\) for any fixed \(\epsilon \in (0,1)\).

It is easy to show that, as soon as \(n=n(d)\) grows like \(d\log d\), \(M_{n,d}\) cannot vanish as the dimension increases. To see this, consider \(\bar\mu\) to be the discrete uniform distribution on the \(d+1\) vertices of a simplex \(\Delta=\mathrm{conv}\{v_1,\ldots,v_{d+1}\}\subset\R^d\), for which it holds that
\[
\frac{\E \vol_d (P_{n,d}^{\bar\mu})}{ \vol_d ( \Delta) }=\P(\{X_1,\ldots,X_n\}=\{v_1,\ldots,v_{d+1}\}).
\]
It is well known from the coupon collector's problem (e.g. \cite{coupon}) that, as \(d\) diverges, such quantity tends to \(1/e\) if \(n= d\log d+o(d)\), so in this regime \(M_{n,d}\) does not vanish. In contrast, if \(n\le (1-\epsilon) d\log d\) for any fixed \(\epsilon\in(0,1)\), the above probability tends to \(0\) and the question of whether or not \(\lim_{d\to\infty} M_{n,d}=0\) remains open.

\section*{Acknowledgments}
We would like to thank Christoph Th\"ale (Bochum) for initiating this collaboration. The work of Zakhar Kabluchko has been supported by the German Research Foundation under Germany’s Excellence Strategy EXC 2044 - 390685587, Mathematics M\"unster: Dynamics - Geometry - Structure. Nicola Turchi is supported by the FNR grant FoRGES (R-AGR-3376-10) at Luxembourg University.

\bibliographystyle{abbrv}
\bibliography{phasetransition}

\end{document}